\documentclass{aic2}

\aicAUTHORdetails{%
  title = {A Multidimensional Rado Theorem}, 
  author = {Aaron Robertson},
  keywords = {Ramsey theory, Rado's Theorem},
}   

\aicEDITORdetails{%
}   

\usepackage{amssymb,latexsym,amsmath,epsfig,amsthm,array,blkarray,booktabs,bigstrut,subcaption} 

 \usepackage{mathrsfs}

\newtheorem{theorem}{Theorem}
\newtheorem{lemma}[theorem]{Lemma}

\newtheorem{observation}[theorem]{Observation}

\theoremstyle{definition}
\newtheorem{definition}[theorem]{Definition}

\newtheorem*{question}{Question}
\newtheorem{remark}[theorem]{Remark}

\newcommand{\thickhline}{%
    \noalign {\ifnum 0=`}\fi \hrule height 1pt
    \futurelet \reserved@a \@xhline
}
\newcolumntype{"}{@{\hskip\tabcolsep\vrule width 2pt\hskip\tabcolsep}}

\def\Zd{\left(\mathbb{Z}^+\right)^d}
\def\Zk{\left(\mathbb{Z}^+\right)^k}
\def\c{\mathbf{c}}
\def\s{\mathbf{s}}
\def\Z{\mathbb{Z}^+}

\newcommand{\xx}[1]{\mathbf x_{#1}}

\begin{document}

\begin{frontmatter}[classification=text]

  \title{A Multidimensional Rado Theorem}
  \author[AR]{Aaron Robertson}

\begin{abstract}
We extend Deuber's theorem on $(m,p,c)$-sets to hold over the
multidimensional positive integer lattices.  This
leads to a multidimensional Rado theorem where we are guaranteed
monochromatic multidimensional points in all finite colorings of
$\Zd$ where the $i^{\mathrm{th}}$ set of coordinates satisfies
the $i^{\mathrm{th}}$ given linear Rado system.
\end{abstract}

\end{frontmatter}




\section{Introduction}

Rado's Theorem \cite{Rado} completely characterizes which linear systems
have the property that any finite coloring of $\Z$ admits a monochromatic
solution to the linear system.  In some sense, this makes such a linear
system ``unbreakable" via partitioning.  This characterization of linear systems
is given in the following definition.

\begin{definition}[Columns Condition]\label{colcond} Let $A$ be an $\ell \times k$ matrix with column
vectors $\c_1, \c_2,\dots,\c_k$. We say that $A$ satisfies the
{\it columns condition} if, after renumbering if necessary, there
exist indices $i_0=0 <i_1< i_2 < \cdots <i_m = k$ such that,
for
$$
 \s_j = \sum_{i=i_{j-1}+1}^{i_j} \c_i, \qquad 1 \leq j \leq m,
$$
the following hold:

\begin{enumerate}
\item $\s_1=\mathbf{0} \in \mathbb{Z}^\ell$;

\item for $2 \leq j \leq m$ we can write $\s_j$ as a linear
combination over $\mathbb{Q}$ of $\c_1,\dots,\c_{i_{j-1}}$.
\end{enumerate}
\end{definition}

With this definition, we can state Rado's Theorem.

\begin{theorem}[Rado's Theorem \cite{Rado}]\label{RadoFullThm} Let $r \in \mathbb{Z}^+$ and let
$A\mathbf{v} = \mathbf{0}$ be a system of equations, where $A$ is a matrix with integer coefficients.  Any $r$-coloring of $\mathbb{Z}^+$ admits
a monochromatic solution to the system if and only if $A$ satisfies the columns condition.
\end{theorem}

A linear system $A\mathbf{v}=\mathbf{0}$ for which $A$ satisfies the columns condition
will be referred to as being a {\it Rado system}.  For example, 
$A=[1\,\,1 \dots 1\,\, -1]$, which corresponds to the single-equation system
$\sum_{i=1}^{k-1}x_i=x_k$, is a Rado system and is often referred to as
the generalized Schur equation (Schur's Theorem, which predates Rado's result, deals with
the case $x_1+x_2=x_3$).

Recently, Balaji, Lott, and Rice \cite{BLR} extended the generalized Schur's Theorem (see \cite{Schur}, \cite{BB}, \cite{RS})  to hold over  
multidimensional positive integer lattices by considering solutions to 
$\sum_{i=1}^{k-1}x_i=x_k$ with ${x_i} \in [1,n]^d$ instead of $[1,n]$, where $[1,n] = \{1,2\dots,n\}$. Can this be extended to different
linear systems?  Moreover, could we possibly have different sets of coordinates satisfy different linear systems
and still guarantee monotonicity?

We do know that van der Waerden's Theorem \cite{vdw}, which states that any finite
coloring of $\Z$ admits arbitrarily long monochromatic arithmetic progressions, has
a multidimensional analog (in \cite{Rado2}, Rado gives credit to Gr\"unwald, the last name
of Gallai at the time; see also \cite{Witt}) along with a multidimensional polynomial analog
\cite{BL}.  Key to the proof of the multidimensional van der Waerden Theorem is the
translation-invariant property of arithmetic progressions.  In this sense, this multidimensional
analog is more geometric and less about the linear system.

Our goal here is about generalizing linear systems to  multidimensional space.
As a motivating question, consider the following.

\begin{question} Does there exist a minimum integer $n$ such that for every $2$-coloring of
$[1,n] \times [1,n]$ there exist $a,b,x,d \in \mathbb{Z}^+$ with $(a,x), (b,x+d),$ and $(a+b,x+2d)$  
all  the same color?
\end{question}

Notice that we are asking for the first coordinates to form a Schur triple and for the second
coordinates to form a 3-term arithmetic progression.  In one dimension,
by considering the system of equations $\{x_1+x_2-x_3=0, x_5-x_4-d=0, x_6-x_5-d=0\}$ (with
$d$ a variable), Rado's Theorem informs
us that we can guarantee a monochromatic Schur triple and monochromatic 3-term arithmetic progression of the same color
under any finite coloring of $\Z$.  However, there does not seem to be a bijection  
$\Z \times \Z \rightarrow \Z$ that preserves the linear system's solutions with 
well-defined colorings; such a bijection would guarantee an answer to our question
in $2$-dimensions via an appeal to Rado's Theorem.

In general, we investigate here which  
linear systems have the Rado property under partitioning, not of
$\Z$, but of $\Zd$.  
To discuss the multidimensional situation, we will use the following notation and language to distinguish
from the classical (1-dimensional) case.

\vskip 5pt
\noindent
{\bf Notation.} For $X \!=\! [x_{ij}]$ a $d \times k$ matrix,
we   refer to   $\mathbf{p}_j = (x_{1j}, x_{2j}, \dots, x_{dj})^\intercal\in\Zd$ as {\it points} and to 
$\xx{i} = (x_{i1}, x_{i2}, \dots, x_{ik}) \in \Zk$ as the {\it $i^{\mathit{th}}$ coordinates vector}.
To generalize a linear system  to address the motivating question above, we
seek to find, for given matrices $A_i$, a set of $k$ monochromatic points $\mathbf{p}_j$ such that
\begin{equation}\label{lvs}
A_i \xx{i}^\intercal={\mathbf{0}} \mbox{ for all $i \in \{1,2,\dots, d\}$}. 
\end{equation}
In other words, the $i^{\mathrm{th}}$ row of $X$ satisfies
the linear system $A_i{\mathbf{v}}=0$.  We will refer to
(\ref{lvs}) as a {\it $d$-dimensional linear vector system}
or simply a {\it linear vector system} if the dimensionality is clear.
In the situation where $A_1=A_2=\cdots=A_d=A$,
the equations in (\ref{lvs}) can be written as $AX^\intercal={\mathbf{0}}$, which we will
refer to as a {\it diagonal linear vector system}.

\vskip 5pt
\noindent
{\bf Example.}  For our motivating question, use
$d=2$ and $k=4$ along with the systems' matrices
$A_1 = 
\begin{bmatrix}
1&1&-1&0\\
\end{bmatrix}
$
and
$A_2 = 
\begin{bmatrix}
-1&1&0&-1\\
0&-1&1&-1\\
\end{bmatrix}
$
to see that we seek    monochromatic 2-dimensional points
$\mathbf{p}_j=[a_j\,b_j]^\intercal$, $1 \leq j \leq 4$, such that
$a_1+a_2=a_3$ (the first coordinates vector) and
$b_1, b_2, b_3$ (the second coordinates vector) are in arithmetic progression.
Note that $a_4$ is a dummy variable to make the lengths of the
coordinates vectors equal and that the
common difference ($b_4$) of the arithmetic progression is not needed to
answer our motivating question, but is necessary for the system.

\vskip 5pt

We start our investigation with   diagonal linear vector systems so that the same system is to be satisfied
by each $i^{\mathit{th}}$ coordinates vector.
Clearly, if we can guarantee a monochromatic solution in $\Z$, when considering points in $\Zd$
we need only consider the points with all coordinates equal since they are in bijection with
$\Z$ and we are (trivially) done.
But can we guarantee solutions if we disallow such solutions? In the case of the generalized Schur's Theorem,
as stated above the answer is yes \cite{BLR}.

Regarding monochromatic solutions to the generalized
Schur equation, we are moving from solutions in $[1,n]$ to solutions in $[1,n]^d$,   and we see that
if we restrict our $r$-coloring of $[1,n]^d$ to just the points on the diagonal,
$\{(i,i,\dots,i): 1 \leq i \leq n\}$, we trivially have monochromatic solutions to $\sum_{i=1}^{k-1}\mathbf{p}_i=\mathbf{p}_k$
 with the $\mathbf{p}_j\in [1,n]^d$ via (the obvious) bijection to $[1,n]$ coupled with Rado's Theorem..
Such a solution is an example of what we will call a degenerate solution (defined later; see Definition \ref{DegenDefn}).
  In \cite{BLR}, via use of the Vandermonde matrix properties and Ramsey's Theorem, it is proved that non-degenerate
solutions exist in the following sense.

\begin{theorem}{{\rm (\cite{BLR})}}\label{BLRThm} Let $d,k,r\in\mathbb{Z}^+$.
Every $r$-coloring of the points in $\Zd$ admits
a monochromatic solution to $\sum_{i=1}^{k-1}\mathbf{p}_i = \mathbf{p}_k$, $\mathbf{p}_j\in \Zd$
for $1 \leq j \leq k$, with
$\mathbf{p}_1,\mathbf{p}_2,\dots,\mathbf{p}_d$ linearly independent (over $\mathbb{Q}$) provided
$d \leq k-1$.
\end{theorem}

As motivation for the main result in this paper, in the next section we extend Balaji, Lott, and Rice's \cite{BLR} proof method, which relies on the
Vandermonde matrix, to show similar results for other equations.
In the third section, we give conditions for the existence of
monochromatic solutions not directly implied by Theorem \ref{RadoFullThm} to systems of
linear vector equations under $r$-colorings of $\Zd$.
We end by answering our motivating question.

 \section{More Motivation}

In this section we extend Theorem \ref{BLRThm} to a larger family of equations via the following
observation, which follows easily from the proof of Theorem
\ref{BLRThm} as found in \cite{BLR}

\begin{observation}\label{Obs} Let $d,k,\ell,r\in\mathbb{Z}^+$ with $d \geq 2$.
Every $r$-coloring of the points in $\Zd$ admits
a monochromatic solution to the linear vector equation 
$$\sum_{i=1}^{k}\mathbf{p}_i = \sum_{i=1}^{\ell}\mathbf{q}_i$$ with  
$\mathbf{p}_i, \mathbf{q}_i\in\Zd$ and
$\mathbf{p}_1, \mathbf{p}_2,\dots, \mathbf{p}_d$ linearly independent (over $\mathbb{Q}$) provided
$d \leq k-1$ and $\mathbf{q}_1, \mathbf{q}_2,\dots, \mathbf{q}_d$ linearly independent provided $d \leq \ell-1$.
Moreover, for such a solution we can have $\{\mathbf{p}_1, \mathbf{p}_2,\dots, \mathbf{p}_{k-1}\} \cap \{\mathbf{q}_1, \mathbf{q}_2 ,\dots, \mathbf{q}_{\ell-1}\} = \emptyset$.
\end{observation}

\begin{proof} Let $n=R(k+\ell;r)$ be the Ramsey number so that every $r$-coloring
of the edges of $K_n$ admits a $K_{k+\ell}$ subgraph with all edges the same color.  Let $\chi$ be
an arbitrary $r$-coloring $[1,n]^d$ .

Consider the complete graph on vertex set
$\{(i,i^2,\dots,i^d): i \in [1,n]\}$.  Color the edge connecting vertex $(i,i^2,\dots,i^d)$ and
$(j,j^2,\dots,j^d)$, with $j>i$, by $\chi((j-i,j^2-i^2,\dots,j^d-i^d))$.  By Ramsey's Theorem
we have a monochromatic subgraph on vertex set
$\{(i_j, i_j^2,\dots,i_j^d): 1 \leq j \leq k+\ell\}$.  Assume that $i_1<i_2<\dots<i_{k+\ell}$ so that, under $\chi$, 
for every pair $s < t$, the points
$(i_t-i_s, i_t^2-i_s^2, \dots,i_t^d-i_s^d)$   are all identically colored. Label these points $P(s,t)$.

Let $\mathbf{p}_1=P(1,2), \mathbf{p}_2=P(2,3),\dots,\mathbf{p}_{k-1}=P(k-1,k), \mathbf{p}_k=P(k,k+\ell)$ and
$\mathbf{q}_1=P(k+1,k+2), \mathbf{q}_2=P(k+2,k+3), \dots, \mathbf{q}_{\ell-1}=P(k+\ell-1,k+\ell), \mathbf{q}_\ell=P(1,k+1)$.  It is easy to check that
$\sum_{i=1}^{k}\mathbf{p}_i = \sum_{i=1}^{\ell}\mathbf{q}_i$ and the argument using the Vandermonde matrix
as found in \cite{BLR} proves the linear independence claim.

To see that $\{\mathbf{p}_1, \mathbf{p}_2,\dots, \mathbf{p}_{k-1}\} \cap \{\mathbf{q}_1, \mathbf{q}_2 ,\dots, \mathbf{q}_{\ell-1}\} = \emptyset$, assume otherwise
and consider $P(s,s+1) = P(k+t,k+t+1)$ for some $s \in [1,k-1]$ and $t \in [1,\ell-1]$. Note that $s+1<k+t$.  Since the first
coordinates are equal (by assumption), we have $i_{s+1}-i_s = i_{k+t+1}-i_{k+t}$.
We now look at the second coordinates.  First, we have
$i^2_{k+t+1} - i^2_{k+t} = (i_{k+t+1}-i_{k+t})(i_{k+t+1}+i_{k+t})$.  Using the equality of
the first coordinates, this means that 
$i^2_{k+t+1} - i^2_{k+t} = (i_{s+1}-i_s)(i_{k+t+1}+i_{k+t}) >(i_{s+1}-i_s)(i_{s+1}+i_s)=i^2_{s+1}-i^2_s$,
so that the second coordinates do not agree.  Hence,  $P(s,s+1) \neq P(k+t,k+t+1)$.
\end{proof}

We now see that there are linear equations other than the generalized Schur equation
for which we can guarantee monochromatic solutions in multidimensional spaces
 not implied by the classical case.
This suggests the question: Can we get a similar result for any Rado system?

\section{Generic Linear Homogeneous Vector Systems}

When considering solutions
to a generic linear homogeneous vector equation $\sum_{i=1}^k a_i \mathbf{v}_i = \mathbf{0}$
we lose Ramsey's Theorem and the Vandermonde matrix as a tool to prove linear independence.
However,  it is not clear that linear independence is the correct property to use in
general.  What we do seek is to show that there are monochromatic solutions in $\Zd$
that cannot be obtained by direct application of Rado's Theorem (and a bijection
with the diagonal of $\Zd$).

\vskip 5pt
\begin{definition}[Degenerate]\label{DegenDefn} For $\mathbf{v}, \mathbf{w} \in \Zd$, let
 $\mathbf{v} \circ \mathbf{w}=(v_1w_1,v_2w_2,\dots,v_dw_d)$ denote the Hadamard product
 (also known as, coincidentally, the Schur product). 
We say a finite set of points $P\subseteq \Zd$ is {\it degenerate} if there exists $\mathbf{v} \in \Zd$ such
that $P \subseteq \{\mathbf{v} \circ (i,i,\dots,i): 1 \leq i \leq n\} $ for some $n \in \Z$.  Otherwise, we say that $P$ is
{\it non-degenerate}.
\end{definition}

Note that degenerate solutions are solutions that can be obtained by an obvious bijection to $[1,n]$ and direct application of Rado's Theorem, while non-degenerate solutions 
do not appear to be.  Hence, we strive to prove that monochromatic non-degenerate solutions exist
in colorings of multidimensional spaces.

We start by proving a multidimensional Deuber theorem.  Toward this end, we will use
the following definition in order to state Deuber's Theorem \cite{Deu}.

\begin{definition}[$(m,p,c)$-set] Let $m,p,c \in \Z$. A set $M \subseteq \Z$ is called an {\it $(m,p,c)$-set} if
there exist {\it generators} $g_1,g_2,\dots,g_m \in \Z$ such that
$$
M = \bigcup_{i=1}^m \left\{cg_i + \sum_{j=i+1}^m \lambda_j g_j: 
\lambda_j \in \mathbb{Z} \cap [-p,p] \mbox{ for } 1 \leq j \leq m\right\},
$$
where we take the empty sum to equal 0.
\end{definition}

Note that we have the integers $\lambda_j$ taking on negative values but that
the set $M$ must be a set of positive integers.

\begin{theorem}[Deuber's Theorem]\label{DeuberThm}
Let $m,p,c,r \in \Z$ be fixed.  Then there exist $M,P,\mu \in \Z$ so
that every $r$-coloring of any   $(M,P,c^\mu)$-set admits
a monochromatic $(m,p,c)$-set.
\end{theorem}
 
 The applicability of Deuber's Theorem comes from the fact that for any $\ell \times m$ matrix $A$
that satisfies the columns conditions, there exist $p,c \in \Z$ such that {\it every}
$(m,p,c)$-set contains a solution to $A\mathbf{v} = \mathbf{0}$.  Hence, Deuber's Theorem
implies one direction of Rado's Theorem since $\Z$ contains every
$(M,P,c^\mu)$-set.

\vskip 5pt
\noindent
{\bf Notation.} When referring to the quantities in Theorem \ref{DeuberThm}
we will say that $(m,p,c)$ is {\it realized} by $M,P,c^\mu$ and write
$(M,P,c^\mu) \vDash_r (m,p,c)$, where we may omit $r$ if the number
of colors in context is clear.
\vskip 5pt
The following easy lemma will prove useful.

\begin{lemma}\label{MPCLemma} The following hold:
\begin{enumerate}
\item Let $r<s$ be positive integers. If $(M,P,c^\mu) \vDash_s (m,p,c)$, then
$(M,P,c^\mu) \vDash_r (m,p,c)$.
\item Let $M'\geq M$ and $P'\geq P$.  If $(M,P,c^\mu) \vDash_r (m,p,c)$, then
$(M',P',c^\mu) \vDash_r (m,p,c)$.
\item Let $t>\mu$.  If $(M,P,c^\mu) \vDash_r (m,p,c)$, then
$(M,c^{t-\mu}P,c^t) \vDash_r (m,p,c)$.
\end{enumerate}
\end{lemma}

 \begin{proof} For (1), note that an $r$-coloring is an $s$-coloring where some colors are not used.
 For (2),  just take a subset of the generators and a subsets of the $\lambda$ values.
 We prove (3) by showing that we always have an $(M,P,c^\mu)$-set contained in any
 $(M,c^{t-\mu}P, c^t)$-set.  Let $g_1,g_2,\dots,g_M$ be generators of an $(M,c^{t-\mu}P, c^t)$-set.
 Set $h_i = c^{t-\mu}g_i$ for $1 \leq i \leq M$ and consider them as generators of an
 $(M,P,c^{\mu})$-set.  It is easy to see that this  $(M,P,c^{\mu})$-set is a subset of the 
  $(M,c^{t-\mu}P, c^t)$-set.  Since every $r$-coloring of any $(M,P,c^{\mu})$-set admits
  a monochromatic $(m,p,c)$-set, we have a monochromatic $(m,p,c)$-set inside any $r$-coloring of an
  $(M,c^{t-\mu}P, c^t)$-set.
 \end{proof}

\subsection{A Multidimensional Deuber Theorem}

There is already a multidimensional analog of Deuber's Theorem due to Bergelson, Johnson, and
Moreira \cite{Berg}.  In their result, one of their goals is to replace the ``$p$" in $(m,p,c)$-sets with
a family of multidimensional polynomial functions and ``$c$" by an additive homomorphism and
letting the generators be vectors. While being  a strong 
multidimensional polynomial generalization of
Deuber's Theorem, it does not   correlate with non-degenerate solutions of a linear vector system in any way that
they present in \cite{Berg} or that
this author could determine.  They do present a generalization applying to systems of equations, which
generalizes coloring $\Z$ to coloring any countable commutative
semigroup with identity 0 (and also with a restriction on how the columns condition is satisfied).  Hence,  their generalization holds over $\left(\mathbb{Z}^{\geq 0}\right)^d$, but not
$\left(\mathbb{Z}^+\right)^d$, which is a small, but quite significant difference.  Another   difference is with their use of idempotent ultrafilters in the Stone-\v{C}ech
compactification of the positive integers, as the existence of such idempotents 
assumes that  every filter on $\mathbb{R}$ can be extended to
an ultrafilter \cite{DT}, a recent result of Di Nasso and Tachtsis   that
is a strictly weaker assumption than the Axiom of Choice (which until their result was needed to prove the existence of
idempotent ultrafilters).

We start by defining multidimensional $(m,p,c)$-sets.

\begin{definition} Let $d,m,p,c\in\Z$.  Let $M_1,M_2,\dots,M_d$ each be an $(m,p,c)$-set.
We call $M_1 \times M_2 \times \cdots \times M_d$ a {\it $d$-dimensional $(m,p,c)$-set}.
\end{definition}

\begin{remark}\label{Rem12} There are less than $(2p+1)^{md}$ elements in a $d$-dimensional
$(m,p,c)$-set.
\end{remark}

\begin{theorem}[A Multidimensional Deuber Theorem]\label{MDeuberThm} Let $d,m,p,c,r \in \Z$.  There exist $M,P,\mu \in \Z$ such
that every $r$-coloring  of any $d$-dimensional $(M,P,c^\mu)$-set admits
a monochromatic $d$-dimensional $(m,p,c)$-set.

\end{theorem}

\begin{proof} Using the result and notation of Theorem \ref{DeuberThm}, we start by giving
notation for the existence of $M,P,$ and $\mu$.  We let
${\mathcal D}_d(m,p,c;r)$ represent the statement
$$
\begin{array}{rl}
{\mathcal D}_d(m,p,c;r): &\mbox{ there exists $M, P, \mu \in\Z$ so that every $r$-coloring of a\, }\\
&\mbox{ $d$-dimensional $(M,P,c^\mu)$-set admits a monochromatic}\\
& \mbox{ $d$-dimensional $(m,p,c)$-set}.
\end{array}
$$

We will prove that ${\mathcal D}_d(m,p,c;r)$ is true for all $d,m,p,c,r\in\Z$ by induction on $d$.
With a slight abuse of notation, we will use 
$$
(M_d(r),P_d(r),c^{\mu_d(r)}) \vDash {\mathcal D}_d(m,p,c;r)
$$
for the quantities that realize ${\mathcal D}_d(m,p,c;r)$.

The base case for our induction, $d=1$, is Theorem \ref{DeuberThm} so
we assume that ${\mathcal D}_d(m,p,c;r)$ is true for all $m,p,c,r$ and will show that ${\mathcal D}_{d+1}(m,p,c;r)$
is true for all $m,p,c,r$.

Let $\chi$ be an $r$-coloring of an arbitrary $(d+1)$-dimensional $(M',P'',c^t;r)$-set,
say $S_1 \times S_2 \times \cdots \times S_{d+1}$, where $M', P'',$ and $t$ are
given below.

Let
$$r' = r^{ (2P_d(r)+1)^{dM_d(r)}} $$
so that
$$
\left(M_1(r'), P_1(r'), c^{\mu_1(r')}\right) \vDash {\mathcal D}_1  (m,p,c; r').
$$

From Lemma \ref{MPCLemma}, with
$M' = \max(M_1(r'),M_d(r))$ and $P'=\max(P_1(r'),P_d(r))$,  
we see that\break $ \left(M', P', c^{\mu_1(r')}\right) \vDash {\mathcal D}_1  (m,p,c; r').$

Let $t = \max(\mu_1(r'), \mu_d(r))$, $\mu = \min(\mu_1(r'), \mu_d(r))$,
and   $P''=c^{t-\mu}P'$.
By  Lemma \ref{MPCLemma} again, we see that
$$
(M',P'',c^t) \vDash {\mathcal D}_1(m,p,c;r'),  
$$
and by applying Lemma \ref{MPCLemma} to each dimension of a $d$-dimensional $(M,P,c^\mu)$-set,
we obtain
\begin{equation}\label{eqn1}
(M',P'',c^t) \vDash {\mathcal D}_d(m,p,c;r).
\end{equation}

We claim that
$$
(M',P'',c^t) \vDash {\mathcal D}_{d+1}  (m,p,c; r).
$$

Define an $r'$-coloring $\gamma$ of $S_1$ as follows.
 For $s \in S_1$, let $\gamma(s)$ be the $r$-coloring of the set
 $$\{\chi(s,x_2,x_3,\dots, x_{d+1}): x_i \in S_i \mbox{ for } 2 \leq i \leq d+1\}$$ (with any
fixed ordering of elements).  From Remark \ref{Rem12}, we see that
this is well-defined (where we may not need all $r'$ colors).

 Since ${\mathcal D}_1(m,p,c;r')$ holds, there exists a $\gamma$-monochromatic $(m,p,c)$-set $T_1$
 so that the sets $$\{\chi((t,x_2,x_3,\dots,x_{d+1}): x_i \in S_i \mbox{ for }2 \leq i \leq d+1\}$$
 are identically $r$-colored for any $t \in T_1$; that is
 $\chi((s,x_2,x_3,\dots,x_{d+1}) = \chi((t,x_2,x_3,\dots,x_{d+1})$ for any $s,t \in T_1$
 and all $x_i\in S_i$.  Let this $r$-coloring be $\tau$.
 Now consider $S_2 \times S_3 \times \cdots \times S_{d+1}$ under $\tau$.  By the inductive hypothesis and Statement (\ref{eqn1}), there exists a $\tau$-monochromatic $d$-dimensional $(m,p,c)$-set
 $T_2 \times T_3 \times \cdots \times T_{d+1}$.  Hence, under $\chi$, our original arbitrary $r$-coloring,
 we see that  $T_1 \times T_2\times \cdots \times T_{d+1}$ is a monochromatic $(d+1)$-dimensional $(m,p,c)$-set, completing
 the inductive step and the proof.
\end{proof}

\subsection{A Multidimensional Rado Theorem}

We now are in a situation where we can apply the following result
 due to Frankl, Graham, and R\"odl \cite{FGR} on the number
 of monochromatic solutions to Rado systems.  Their method is to choose
 $n$ much larger than $M$ so that there are many $(M,P,c^\mu)$-sets in
 $[1,n]$.  By considering generators $g_1, g_2,\dots, g_{M}$ with
$g_i \equiv (2P+1)^i \pmod{(2P+1)^M}$, they prove the following.
 
 \begin{theorem}[\cite{FGR}] \label{FGRThm} Let $A\mathbf{v}=\mathbf{0}$ be a linear  system with $A$ an $\ell \times k$ matrix of rank $\ell$ satisfying the columns condition.  For any $r \in \Z$, there exists a constant $c=c(A,r)>0$ such that any $r$-coloring of $[1,n]$
 contains at least $cn^{k-\ell}(1+o(1))$ monochromatic solutions.
 \end{theorem}

Applying Theorem \ref{FGRThm} to each coordinate of a $d$-dimensional $(M,P,c^\mu)$-set, we can
now state the following result.

\begin{theorem}\label{NumSolThm} Consider the linear vector  system $A\mathbf{v}=\mathbf{0}$ with $A$ an $\ell \times k$ scalar matrix of rank $\ell$ satisfying the columns condition.    Let $r \in \Z$. There exists a constant $c = c(A,d,r)>0$ such that
any $r$-coloring of $[1,n]^d$ admits at least $cn^{d(k-\ell)}(1+o(1))$ monochromatic solutions to $A\mathbf{v}=\mathbf{0}$.
\end{theorem}

\begin{remark}\label{LastRem}
For a linear  system $A\mathbf{v}=\mathbf{0}$ with $A$ an $\ell \times k$ matrix of full rank satisfying the columns condition 
there are $\Omega(n^{k-\ell})$ solutions to the system in $[1,n]$.  Since there are $n^d$ points in $[1,n]^d$, 
recalling Definition \ref{DegenDefn} we see that there are  
$O(n^d \cdot n^{k-\ell}) = O(n^{k -\ell+d})$ degenerate monochromatic solutions in any $r$-coloring of $[1,n]^d$.
\end{remark}

Our main result is next and shows that in most situations we can guarantee monochromatic solutions that are
not directly implied   by Rado's Theorem.

\begin{theorem}[A Multidimensional Rado Theorem]\label{MRado}  Let $r,d \in \mathbb{Z}^+$. Consider the linear vector system $A\mathbf{v}=\mathbf{0}$.  
If $A$ is an $\ell \times k$ matrix of rank $\ell \leq k-2 $ that satisfies the columns condition, then every $r$-coloring of $\Zd$ admits a
monochromatic non-degenerate solution.
\end{theorem}

\begin{proof} Let $n \in \Z$ be sufficiently large.
From Remark \ref{LastRem} we see that for some $c_1>0$ there are at most $c_1n^{k-\ell+d}(1+o(1))$ monochromatic
degenerate solutions in any $r$-coloring of $[1,n]^d$.  From Theorem \ref{NumSolThm} we have a total of at
least $c_2n^{d(k-\ell)}(1+o(1))$ monochromatic solutions.  Hence, we have at least
$c_3n^{(d-1)(k-\ell)-d}(1+o(1))$ non-degenerate monochromatic solutions.  
By ensuring that the exponent is strictly greater than 0, we see that
for $d>2$, it suffices to have $k -\ell \geq 2$; however,
for $d=2$ the counting argument requires $k - \ell \geq 3$ to have non-degenerate monochromatic solutions.
Hence, to complete the proof, we will show via a different argument that for $d=2$ and $k-\ell=2$ we can also guarantee non-degenerate monochromatic solutions. 

For a $(k-2)\times k$ linear vector system in two dimensions, we will enumerate degenerate solutions.
From Theorem \ref{NumSolThm}, we know there are $c_1n^{4}(1+o(1))$ monochromatic solutions in
any $r$-coloring of $[1,n]^2$ for some $c_1>0$.  We will complete the argument by showing that there are $o(n^{4})$ (not necessarily
monochromatic)  degenerate solutions.

In a $(k-2)\times k$ linear system (in one dimension) of full rank, we have 2 free variables, from which the values of all other variables
are determined by the system.  Let $x$ and $y$ be the free variables (for the one dimension situation) and
note that when extended to more dimensions, all coordinates of the variable point corresponding to the free variable
are free.  Hence, we can consider $(x_1,x_2)$ and $(y_1,y_2)$ to be the free variables.

For a degenerate solution in two dimensions we have $(x_1,x_2) = \frac{p}{q}(y_1,y_2)$ for some $p,q \in \Z$, where
we may assume that $\gcd(p,q)=1$.
Let $q$ be fixed.  Then both $y_1$ and $y_2$ must be divisible by $q$.
Hence, there are at most $\left(\frac{n}{q}\right)^2$ choices for $(y_1,y_2)$.
Now, we require $p \leq nq$ to ensure that $x_1,x_2 \leq n$. Thus, for each $q$, we have at most $\frac{n^3}{q}$ possibilities
for $(x_1,x_2)=\frac{p}{q}(y_1,y_2)$.  Summing over all possible values of $q$ gives
at most
$$
\sum_{q=1}^n \frac{n^3}{q} = n^3\sum_{q=1}^n \frac{1}{q} < 2n^3 \log n
$$
possibilities for $(x_1,x_2)$ and $(y_1,y_2)$.  Since the values of the
other variables are determined by the free variables' values, we find that, since $2n^3 \log n = o(n^4)$, there must be a non-degenerate monochromatic solution 
for $d=2$ and $k-\ell=2$.
\end{proof}

\subsection{More Flexible Multidimensional Deuber and Rado Theorems} 

Recall that for $\mathbf{v},\mathbf{w} \in \Zd$, we defined
 $\mathbf{v} \circ \mathbf{w}=(v_1w_1,v_2w_2,\dots,v_dw_d)$ and consider
 $$
 \begin{bmatrix}
 1\\
 1\\
 \end{bmatrix}
 \circ
  \begin{bmatrix}
 x_1\\
 x_2\\
 \end{bmatrix}
 +
  \begin{bmatrix}
 1\\
 1\\
 \end{bmatrix}
 \circ
  \begin{bmatrix}
y_1\\
y_2\\
 \end{bmatrix}
 =
   \begin{bmatrix}
 1\\
2\\
 \end{bmatrix}
 \circ
  \begin{bmatrix}
z_1\\
z_2\\
 \end{bmatrix}.
 $$
In this vector equation we seek points $(x_1,x_2), (y_1,y_2)$, and $(z_1,z_2)$ of the same color
such that $x_1, y_1, z_1$ is a Schur triple $(x_1+y_1=z_1$) while $x_2,z_2,y_2$ is a 3-term arithmetic progression.
Can we guarantee that this occurs?

In this subsection we show that multidimensional
$(m,p,c)$-sets offer more flexibility than what is stated in Theorem \ref{MRado}.
In particular, we will show that the coordinates   can be required to
satisfy separate systems.  We need a definition with some notation to be more explicit in this situation.

\begin{definition}[$(\mathbf{m},\mathbf{p},\mathbf{c})$-set] Let $d \in \Z$ and let $\mathbf{m},\mathbf{p},\mathbf{c} \in \Zd$.
We say that $S_1 \times S_2 \times \cdots \times S_d$ is an {\it $(\mathbf{m},\mathbf{p},\mathbf{c})$-set} if
$S_j$ is an $(m_j,p_j,c_j)$-set for every $j\in[1,d]$.
\end{definition}

Following the proof of Theorem \ref{MDeuberThm}, the following generalization can be shown.

\begin{theorem}\label{OffDiagDeub} Let $d,r \in \Z$ and let $\mathbf{m},\mathbf{p},\mathbf{c} \in \Zd$.  There exists $\mathbf{M},\mathbf{P},\mathbf{C} \in \Zd$ such
that every $r$-coloring  of any $(\mathbf{M},\mathbf{P},\mathbf{C} )$-set admits
a monochromatic  $(\mathbf{m},\mathbf{p},\mathbf{c})$-set.
\end{theorem}

\begin{proof} Using the notation from the proof of Theorem \ref{MDeuberThm}, we offer a terse proof as
it is essentially the same proof. We let
${\mathcal D}_d(\mathbf{m},\mathbf{p},\mathbf{c};r)$ represent the statement
$$
\begin{array}{rl}
{\mathcal D}_d(\mathbf{m},\mathbf{p},\mathbf{c};r): &\mbox{ there exists $\mathbf{M},\mathbf{P},\mathbf{C}  \in\Zd$ so that every $r$-coloring of a\, }\\
&\mbox{ ($d$-dimensional) $(\mathbf{M},\mathbf{P},\mathbf{C} )$-set admits a monochromatic}\\
& \mbox{ ($d$-dimensional) $(\mathbf{m},\mathbf{p},\mathbf{c})$-set}.
\end{array}
$$

Using induction on $d$, with $d=1$ being Deuber's Theorem, we will use 
$$
(\mathbf{M}(d;r),\mathbf{P}(d;r),\mathbf{C}(d;r)) \vDash {\mathcal D}_d(\mathbf{m},\mathbf{p},\mathbf{c};r)
$$
Let $\chi$ be an $r$-coloring of an arbitrary $(d+1)$-dimensional $(\mathbf{M'},\mathbf{P'},\mathbf{C'};r)$-set,
say $S_1 \times S_2 \times \cdots \times S_{d+1}$, where $\mathbf{M'},\mathbf{P'}$, and $\mathbf{C'}$ are
given below.

Since there are at most $e=\prod_{i=1}^d (2P_i(r)+1)^{M_i(r)}$ elements in an
$(\mathbf{M}(d;r), \mathbf{P}(d;r),\break \mathbf{C}(d;r))$-set, let $r'=r^e$
so that
$
\left(M(1;r'), P(1;r'), C(1;r')\right) \vDash {\mathcal D}_1  (m_1,p_1,c_1; r').
$
By Lemma \ref{MPCLemma}, since $r'>r$, we have $\left(M(1;r'), P(1;r'), C(1;r')\right) \vDash {\mathcal D}_1  (m_1,p_1,c_1; r)$.
Let $\mathbf{M'} = M(1;r') \times \mathbf{M}(d;r)$, $\mathbf{P'} = P(1;r') \times \mathbf{P}(d;r)$, and
$\mathbf{C'} = C(1;r') \times \mathbf{C}(d;r)$.

Using the $r'$-coloring $\gamma$ from the proof of  Theorem \ref{MDeuberThm} and the argument therein completes the inductive step and proof.
\end{proof}

We can consider Theorem \ref{OffDiagDeub} to be an off-diagonal multidimensional Deuber's Theorem.
This opens up monochromatic sets of points in $\Zd$ with interesting mixed Rado properties:
By applying Theorems \ref{FGRThm} and \ref{OffDiagDeub} we get the following result.

\begin{theorem}[An Off-Diagonal Multidimensional Rado Theorem]\label{OffDiagMRado} Let $d,r \in \Z$.  For $1 \leq j \leq d$, let $A_j$ be an $\ell_j \times k$ matrix of rank $\ell_j$ satisfying the columns condition.    There exists a constant $c = c(A_1,\dots,A_d,d,r)>0$ such that
any $r$-coloring of $[1,n]^d$ admits at least $cn^{\sum_{i=1}^d(k-\ell_i)}(1+o(1))$ monochromatic sets of $k$ points
$\mathbf{p}_i$  such that row $\mathbf{r}_j$ of the $d \times k$ matrix $B=[\mathbf{p}_1\,\mathbf{p}_2\,\dots, \mathbf{p}_k]$
satisfies
$
A_j \mathbf{r}_j^\intercal = \mathbf{0}
$
for each $j \in [1,d]$; that is, the $j^{\mathit{th}}$ coordinates vector of
$B$ satisfies the $j^{\mathit{th}}$ system.
\end{theorem}

\begin{remark}
Systems with different numbers of variables can all be made to have $k$ variables, for some $k$, by introducing dummy variables.
\end{remark}

Note that degenerate solutions no longer make sense if the systems are not equivalent, but we can now consider coloring
the points in multidimensional space and guaranteeing the existence of many interesting monochromatic configurations, like
in the next example.

\vskip 5pt
\noindent
{\bf Example.} For any $r$-coloring of $\mathbb{Z}^3$ there exist
$a,d,x_1,x_2,x_3,y_1, y_2, y_3 \in \mathbb{Z}^+$
such that the points
$$
\begin{pmatrix}
x_1\\
y_1\\
a\\
\end{pmatrix},
\begin{pmatrix}
x_2\\
y_2\\
a+d\\
\end{pmatrix},
\begin{pmatrix}
x_3\\
y_3\\
a+2d\\
\end{pmatrix},
\begin{pmatrix}
x_1+x_2+x_3\\
y_3+y_2-y_1\\
a+3d\\
\end{pmatrix}.
$$
are monochromatic.  Note that the first coordinates vector satisfies the generalized Schur equation,
the second coordinates vector satisfies $u-v = w-x$, and the third coordinates vector
are a 4-term arithmetic progression.

\vskip 5pt

By applying Theorem \ref{OffDiagMRado} we can now answer, in particular, our motivating question at the beginning of
this article about Schur triples and 3-term arithmetic progressions in the affirmative.  In fact, a simple
backtracking computer calculation (or tedious pencil and paper work)
provides us with the following result.

 \begin{theorem}
The minimum integer $n$ such that for every $2$-coloring of
$[1,n] \times [1,n]$ there exist $a,b,x,d \in \mathbb{Z}^+$ with $(a,x), (b,x+d),$ and $(a+b,x+2d)$  
all   the same color, is $n=9$.
 \end{theorem}

\vskip 20pt
\noindent
{\bf Acknowledgement.} The author thanks Tom Brown, Bruce Landman, and Neil Hindman for helpful comments in earlier drafts.

\bibliographystyle{amsplain}

\begin{thebibliography}{10}\footnotesize



\bibitem{BLR} V. Balaji, A. Lott, and A. Rice, Schur's theorem in integer lattices, {\it Integers}
{\bf 22} (2022), \#A62, 8pp.

\bibitem{BB} A. Beutelspacher and W. Brestovansky, Generalized
Schur Numbers, {\it Lecture Notes in Mathematics},
{\bf 969} (1982), 30-38.

\bibitem{BL} V. Bergelson, and A. Leibman,
Polynomial extensions of van der Waerden’s and Szemeredi’s theorems,
 Journal of AMS {\bf 9} (1996),  725-753.


\bibitem{Berg} V. Bergelson, J. H. Johnson Jr., and J. Moreira, New polynomial and
multidimensional extensions of classical partition results,
{\it J. Combin. Theory, Series A} {\bf 147} (2017), 119-154.

\bibitem{Deu} W. Deuber, Partitionen und lineare gleichungssysteme,
{\it Math. Zeitschrift} {\bf 133} (1973), 109-123.

\bibitem{DT} M. Di Nasso and E. Tachtsis, Idempotent ultrafilters without
Zorn's Lemma, {\it Proc. A.M.S} {\bf 146} (2018), 397-411.

\bibitem{FGR} P. Frankl, R. Graham, and V. R\"odl,
Quantitative theorems for regular systems of equations,
{\it J. Combin. Theory, Series A} {\bf 47} (1988), 246-261.

\bibitem{Rado} R. Rado, Studien zur kombinatorik, {\it Math. Z.} {\bf 36} (1933), 424-480.

\bibitem{Rado2} R. Rado, Note on combinatorial analysis, {\it Proc. London Math. Soc.} {\bf 48} (1943), 122-160.

\bibitem{RS} A. Robertson and D. Schaal, Off-diagonal generalized Schur numbers,
{\it Adv. Applied Math.} {\bf 26} (2001), 252-257.


\bibitem{Schur} I. Schur, \"{U}ber die kongruenz $x^m + y^m = z^m \pmod{p}$,
{\it Jahresbericht der Deutschen Mathematiker-Vereinigung} {\bf 25} (1916), 114-117.

\bibitem{Witt} E. Witt, Ein kombinatorischer satz der elementargeometrie,
{\it Math. Nach.} {\bf 6} (1952), 261-262.

\bibitem{vdw} B. L. van der Waerden, Beweis einer Baudetschen vermutung,
{\it Nieuw Archief voor Wiskunde} {\bf 15} (1927), 212-216.

\end{thebibliography}

\begin{aicauthors}
\begin{authorinfo}[AR]
Aaron Robertson\\
Department of Mathematics\\
Colgate University\\
Hamilton, New York\\
arobertson@colgate.edu\\
{\tt http://math.colgate.edu/$\sim$aaron}
\end{authorinfo}
\end{aicauthors}

\end{document}